\documentclass{amsart}
\usepackage{amssymb}
\usepackage{amsfonts}
\usepackage{amsthm}
\usepackage[noadjust]{cite}
\usepackage{mathtools}
\usepackage{xcolor}
\usepackage{graphicx}
\usepackage{caption}
\usepackage{subcaption}

\calclayout

\usepackage{esint}

\numberwithin{equation}{section}
\theoremstyle{plain}
\newtheorem{theorem}[equation]{Theorem}

\newtheorem{lemma}[equation]{Lemma}

\theoremstyle{remark}
\newtheorem{remark}[equation]{Remark}

\theoremstyle{definition}

\newtheorem{question}[equation]{Question}
\newtheorem*{question*}{Question}

\usepackage{datetime} 

\usepackage{graphicx}

\usepackage{url}

\title[non-embeddability of the slit carpet]{Regular mappings and non-existence of Bi-Lipschitz embeddings for Slit carpets}

\author{Guy C. David}
\address{Department of Mathematical Sciences\\ Ball State University, Muncie, IN 47306}
\email{gcdavid@bsu.edu}

\author{Sylvester Eriksson-Bique}
\address{Department of Mathematics, University of California, Los Angeles, Box 95155, Los Angeles, CA, 90095-1555}
\email{syerikss@math.ucla.edu}

\date{\today}


\newcounter{prob}
\newcommand{\N}{\ensuremath{\mathbb{N}}}
\newcommand{\M}{\ensuremath{\mathbb{M}}}
\newcommand{\R}{\ensuremath{\mathbb{R}}}

\DeclareMathOperator{\diam}{diam}

\newcommand{\defeq}{\mathrel{\mathop:}=}

\def\XXint#1#2#3{{\setbox0=\hbox{$#1{#2#3}{\int}$ }
\vcenter{\hbox{$#2#3$ }}\kern-.58\wd0}}

\begin{document}

\begin{abstract}
We prove that the ``slit carpet'' introduced by Merenkov does not admit a bi-Lipschitz embedding into any uniformly convex Banach space. In particular, this includes any space $\R^n$, but also spaces such as $L^p$ for $p \in (1,\infty)$. This resolves Question 8 in the 1997 list by Heinonen and Semmes.
\end{abstract}

\maketitle

\section{Introduction}

In 1997, in the early days of the field now called ``analysis on metric spaces'', Heinonen and Semmes \cite{HS} posed a list of ``Thirty-three yes or no questions about mappings, measures, and metrics,'' which have gone on to be quite influential. A number of these questions have been solved since the publication of this list, but many remain open.

In this paper, we resolve, in the negative, Question 8 from that list:

\begin{question}[\cite{HS}, Question 8]\label{q:HS}
\textit{If an Ahlfors regular metric space admits a regular map into some Euclidean space, then does it admit a bi-Lipschitz map into another, possibly different, Euclidean space?}
\end{question} 

Precise definitions for the relevant terms in the question will be given in Section \ref{sec:notation} below. For now, a bi-Lipschitz map is simply an embedding that preserves distances up to a constant factor, and a regular map is a map that ``folds'' a metric space in a certain quantitative, $N$-to-$1$ manner. Thus, informally, Question \ref{q:HS} asks: if a metric space can be quantitatively \textit{folded} to fit into some Euclidean space, can it be quantitatively \textit{embedded} into some Euclidean space?

We give an explicit example showing that the answer to Question \ref{q:HS} (Question 8 of \cite{HS}) is ``no''. In fact, our example is a space that has already appeared in the literature as an interesting example in a different context. This is the ``slit carpet'' studied by Merenkov in \cite{Me}, which we denote $\M$. We postpone a formal definition of $\M$ to later in the paper, but the idea is the following. Start with the unit square $Q_0 = [0,1]^2$ in the plane. Cut a vertical ``slit'' in the square along the segment $\{\frac{1}{2}\} \times [\frac{1}{4},\frac{3}{4}]$. Now subdivide $Q_0$ into its four natural dyadic subsquares, and similarly remove a vertical slit of of half the side length from the center of each sub-square. Repeating this process in all dyadic squares of all scales in $Q_0$ leaves us with a remaining set $A\subseteq Q_0$, the complement of the countably many slits, which we equip with the \textit{shortest path metric}. The completion of $A$ in this metric is the slit carpet $\M$. See Figure \ref{fig:slitcarpet} below. (We give a slightly more careful definition of $\M$ as a certain limit below.) The name ``slit carpet'' comes from the fact that $\M$ is homeomorphic to the classical planar Sierpi\'nski carpet \cite[Lemma 2.1]{Me}. For an example of how $\M$ serves as a counterexample, see \cite{MW}, where the authors show that there is no quasisymmetric embedding $f: \M \to \R^2$.

\begin{figure}
	\centering
		\includegraphics[scale=0.2]{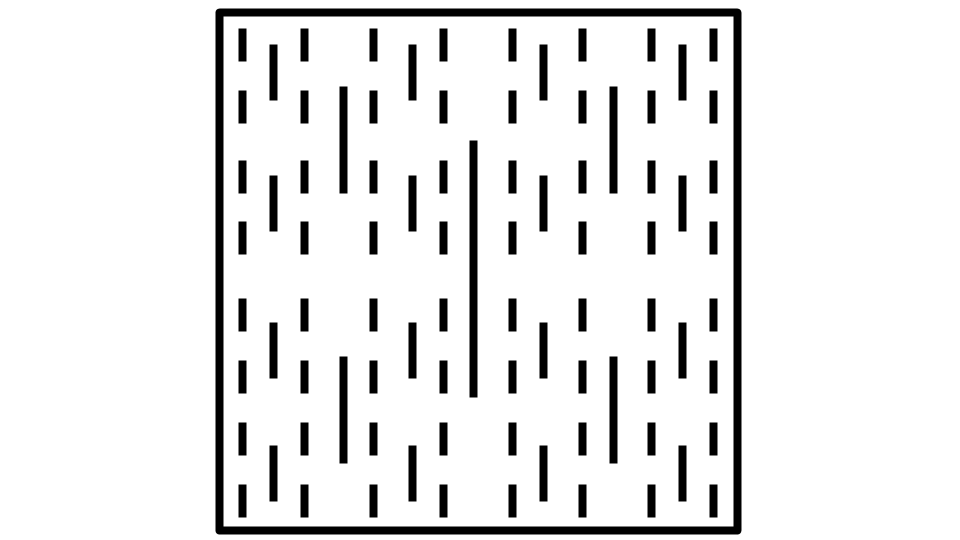}
		\caption{An approximation to the slit carpet $\M$}
	\label{fig:slitcarpet}
\end{figure}

Merenkov defined and studied $\M$ in the context of quasisymmetric mappings, showing that the double of $\M$ along its boundary gives the first known example of a space which is ``quasisymmetrically co-Hopfian'' for metric rather than topological reasons. In the course of defining $\M$, Merenkov observed that $\M$ is Ahlfors regular and that the natural map $\pi\colon \M \rightarrow Q_0 \subseteq \R^2$ given by ``collapsing the slits'' (which we define precisely in Section \ref{sec:notation}) is a regular mapping in the above sense, so $\M$ admits a regular mapping into some Euclidean space. 

However, we prove the following:

\begin{theorem}\label{main-theorem}
There is no bi-Lipschitz map $f\colon \mathbb{M} \to B$, for any uniformly convex Banach space $B$.
\end{theorem}
In particular, $\M$ admits no bi-Lipschitz embedding into any $\R^n$, answering Question \ref{q:HS}, but also no such embedding into any $L^p$ space for $p\in (1,\infty)$.

The proof of Theorem \ref{main-theorem} is based on a technique appearing in seminal work of Burago-Kleiner \cite{BK}. This scheme begins by identifying a certain extremal pair of points, where the Lipschitz constant of the supposed embedding is nearly achieved. In the vicinity of such an extremal pair, with the aid of approximation, the mapping $f$ has to behave roughly linearly in this direction. Finally, this near-linearity is used to contradict some setting-dependent non-linear or non-Euclidean behavior. In our case, the non-Euclidean behavior arises from the presence of the slits at all locations and scales, which leads to a contradiction to the supposed bi-Lipschitz behavior of the map. 

Interestingly, the idea of looking for maximal or almost-maximal directional expansion for Lipschitz functions appears in other contexts as well, namely questions of finding points of differentiability in small sets. See \cite{Fitzpatrick, Preiss}.

We conclude the introduction with two remarks concerning Question \ref{q:HS} and an open question.

\begin{remark}
A 2005 paper by Movahedi-Lankarani and Wells \cite[VII (p. 262)]{MLW} mentions that Tomi Laakso informed those authors that he had an example showing that the answer to Question \ref{q:HS} is ``no'', and in fact has the same property we prove for $\M$ in Theorem \ref{main-theorem}. So far as we know, Laakso's example has never appeared in print, and we do not know if his example or his proof is the same as ours, though our proof certainly owes some ideas to \cite{Laakso}.
\end{remark}

\begin{remark}
One way of seeing the difficulty in Question \ref{q:HS} is to observe that a now-standard technique of proving non-embedding theorems for metric spaces, Cheeger's differentiation theory, cannot possibly provide a ``no'' answer to the question. Cheeger's theory \cite{Ch} endows certain metric spaces, the so-called PI spaces, with a type of ``measurable differentiable structure''. Using this theory, Cheeger showed that PI spaces whose blowups, in the pointed Gromov-Hausdorff sense, are not bi-Lipschitz homeomorphic to Euclidean spaces cannot admit bi-Lipschitz embeddings into any $\R^n$, work which was later generalized by other authors \cite{CK, CK_Banach, CKS, GCD}. This provides a unified approach to many non-embedding theorems, including that of the Heisenberg group \cite{CK_Banach} and Laakso spaces \cite{CK_PI}

On the other hand, it also follows from Cheeger's theory that an Ahlfors regular PI space whose blowups are not bi-Lipschitz homeomorphic to Euclidean spaces cannot even admit a regular map into any Euclidean space. (This follows from the above remarks, the measure-preservation property of regular mappings, and \cite[Theorem 1.5]{GCDKin}.) So Cheeger's differentiation theory is not directly useful for answering Question \ref{q:HS}.
\end{remark}

Of course, a number of interesting Banach spaces do not fit into the uniformly convex framework. For embedding questions, the most interesting of these are probably $\ell^1$ and $L^1$. 

\begin{question}\label{q:l1}
Does $\M$ admit a bi-Lipschitz embedding into $\ell^1$? Into $L^1?$
\end{question}
One way of approaching the question of an $L^1$ embedding is to see if $\M$ has ``Lipschitz dimension 1'' in the sense of \cite{CK}, which by Theorem 1.7 of that paper would force it to be bi-Lipschitz embeddable in $L^1$. Note that by \cite[Lemma 8.9]{GCDLip}, $\M$ has Lipschitz dimension $\leq 2$. However, we do not discuss Lipschitz dimension or Question \ref{q:l1} further in this paper.

\subsection*{Acknowledgments}
G.~ C.~ David was partially supported by the National Science Foundation under Grant No. DMS-1758709. S.~ Eriksson-Bique was partially supported by the National Science Foundation under Grant No. DMS-1704215.

\section{Notation and constructions}\label{sec:notation}

\subsection{Metric spaces and Banach spaces}
Our notation is fairly standard. If $X$ is a metric space, we denote its metric by $d$ unless otherwise specified. The diameter of a metric space is 
$$\diam(X) = \sup\{d(x,y):x,y\in X\}.$$
Open and closed balls in $X$ are denoted $B(x,r)$ and $\overline{B}(x,r)$, respectively.

If $X$ is a metric space and $d>0$, the $d$-dimensional Hausdorff measure on $X$ is defined by
$$ \mathcal{H}^d(E) = \lim_{\delta\rightarrow 0} \inf \sum_{B\in\mathcal{B}} (\diam(B))^d,$$
where the infimum is over all covers $\mathcal{B}$ of $E$ by sets of diameter at most $\delta$. (See, e.g., \cite[Section 8.3]{He}.) Various other standardizations of Hausdorff measure exist that differ from this one by multiplicative constants.

The Hausdorff dimension of a space $X$ is $\inf\{ d: \mathcal{H}^d(X)=0\}$. A stronger, scale-invariant version of having Hausdorff dimension $d$ is the notion of \textit{Ahlfors $d$-regularity}, used in Question \ref{q:HS}. A metric space $X$ is Ahlfors $d$-regular if there are constants $C>c>0$ such that
$$ cr^d \leq \mathcal{H}^d(\overline{B}(x,r)) \leq C r^d \text{ for all } r\leq \diam(X).$$

For Theorem \ref{main-theorem}, we also need to introduce the Banach space property of \textit{uniform convexity}. Recall that a Banach space $B$ is uniformly convex, if for every $\epsilon>0$, there exists a $\delta>0$ such that for any $x,y\in B$ with $\|x\|=\|y\| = 1$ and $\|x-y\| \geq \epsilon$, we have $$\left\|\frac{x+y}{2}\right\| \leq 1 - \delta.$$

\subsection{Lipschitz, bi-Lipschitz, and regular mappings}

Three basic classes of mappings are used in the rest of the paper. A mapping $f\colon X \rightarrow Y$ between metric spaces is called \textit{Lipschitz} if there is a constant $L>0$ such that
$$ d(f(x),f(y)) \leq L d(x,y) \text{ for all } x,y\in X.$$
It is called \textit{bi-Lipschitz} if there are constants $b, L>0$ such that
$$bd(x,y) \leq d(f(x),f(y) \leq Ld(x,y)$$
for all $x,y \in X$. The smallest $L$ and largest $b$ satisfying this are called the Lipschitz and lower-Lipschitz constants for $f$. The pair $(b,L)$ will be referred to as the bi-Lipschitz constants of $f$.

Lastly, a Lipschitz map $g\colon X \rightarrow Y$ is called \textit{regular} if there is a constant $C>0$ such that, for every ball $B\subseteq Y$ of radius $r$, the pre-image $g^{-1}(B)$ can be covered by at most $C$ balls of radius $Cr$.

In particular, regular mappings are always at most $C$-to-$1$, but the definition implies more than this. Regular mappings were introduced by David and Semmes \cite[Definition 12.1]{DS} as a kind of intermediate notion between Lipschitz and bi-Lipschitz mappings. One nice way in which regular mappings generalize bi-Lipschitz mappings is that they preserve the $d$-dimensional measure of all subsets, up to constant factors \cite[Lemma 12.3]{DS}.

\subsection{Construction of the slit carpet}
We now follow \cite[Section 2]{Me} to give a more careful definition of the slit carpet than that in the introduction. Though we use our own notation, the reader may wish to look at \cite{Me} for more details. Generalizations of this construction have also recently been studied by Hakobyan \cite{Hak}.

Let $Q_0=[0,1]^2$ be the unit cube in $\R^2$. Let $\mathcal{D}_n$ be the collection of dyadic sub-cubes of $Q_0$ at level $n$, that is cubes 
$$D_{l,k} = [l2^{-n},(l+1)2^{-n}] \times [k2^{-n}, (k+1)2^{-n}]$$
for $l,k \in \N \cap [0,2^n-1].$  The collection of all such dyadic cubes is denoted $\mathcal{D}$.

For each $D_{l,k}^n \in \mathcal{D}_n$, consider the points 
$$b_{l,k}^n = ((2l+1)2^{-n-1},(4k+1)2^{-n-2}) \text{ and } t_{l,k}^n=((2l+1)2^{-n-1},(4k+3)2^{-n-2}).$$

These define a vertical segment
$$s_{l,k}^n = [b_{l,k}^n,t_{l,k}^n] \subseteq Q_0,$$
which we call a \textit{slit} of level $n$. Note that $\diam(s_{l,k}^n) = 2^{-n-1}$. We call the set 
$$s_{l,k}^n \setminus \{b_{l,k}^n, t_{l,k}^n\}$$
the \textit{interior} of the slit $s_{l,k}^n$.

Define now $M_0 = Q_0 = [0,1]^2$, and, iteratively, $M_{k+1}=M_k \setminus \bigcup_{a,b} s_{a,b}^{k}$. (See Figure \ref{fig:carpets}.) Then define $\mathbb{M}_k$ as the completion of $M_k$ with respect to the shortest path metric $d_k$ on $M_k$. (We will continue to call this new complete metric on $\M_k$ by $d_k$.) This amounts to ``cutting along'' each slit of level $k$ or lower. Note that the $d_k$-diameter of each $\M_k$ is bounded by $3$.

\begin{figure}
\centering
\begin{subfigure}{0.3\textwidth}
  \centering
  \includegraphics[width=\linewidth]{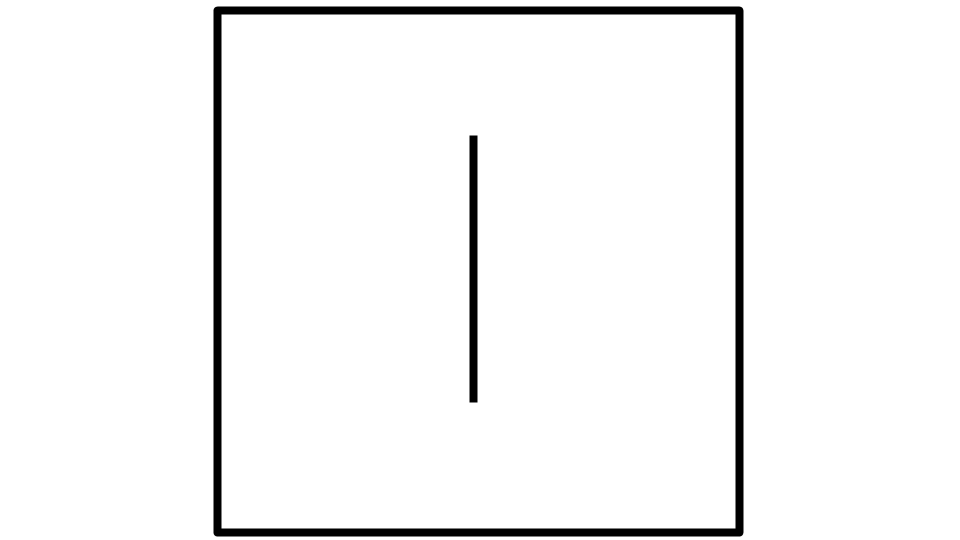}
  \caption{$M_1$}
  \label{fig:carpet1}
\end{subfigure}
\begin{subfigure}{0.3\textwidth}
  \centering
  \includegraphics[width=\linewidth]{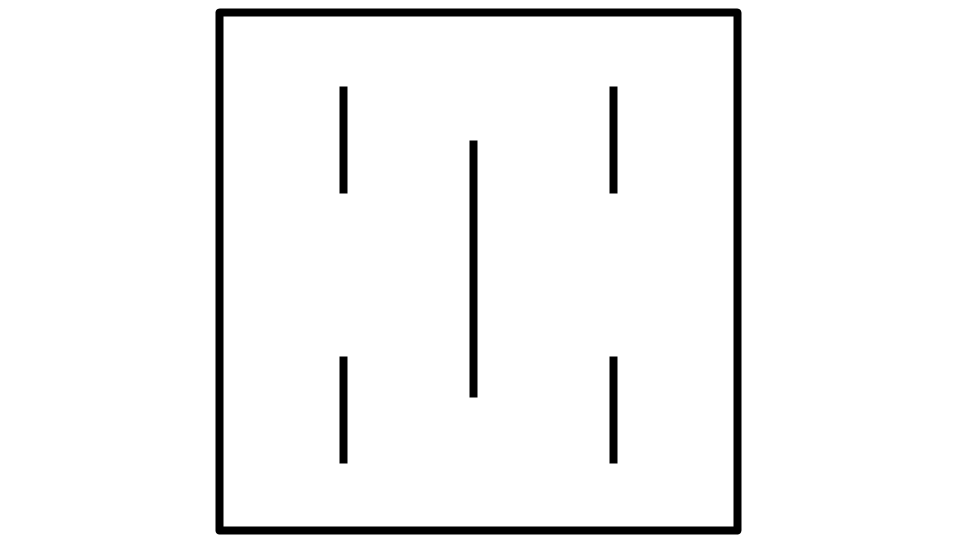}
  \caption{$M_2$}
  \label{fig:carpet2}
\end{subfigure}
\begin{subfigure}{0.3\textwidth}
  \centering
  \includegraphics[width=\linewidth]{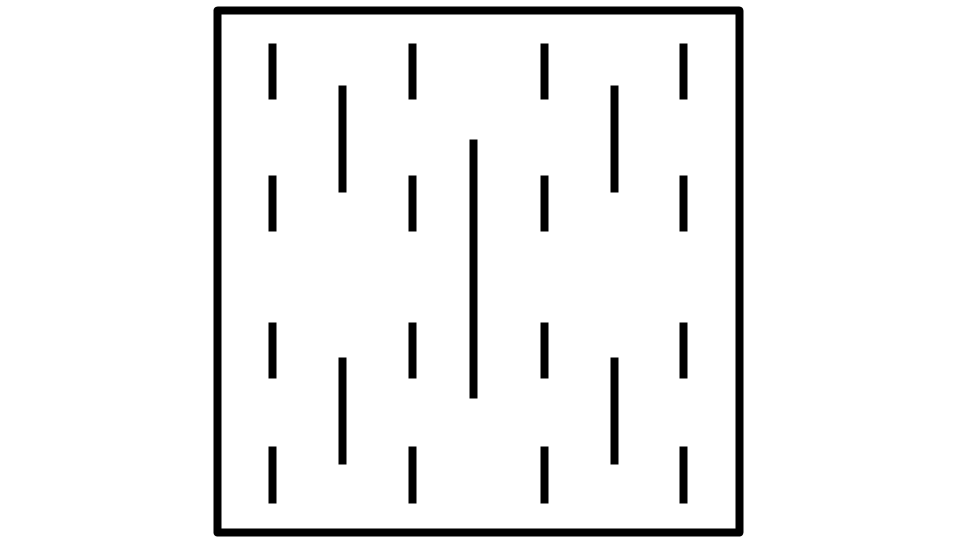}
  \caption{$M_3$}
  \label{fig:carpet3}
\end{subfigure}
\caption{The first three $M_i$.}
\label{fig:carpets}
\end{figure}

As observed by Merenkov, for each $k\leq j$, there is a natural $1$-Lipschitz mapping $\pi_{j,k}\colon \M_j \rightarrow \M_k$ obtained by identifying opposing points on slits of levels greater than $j$ corresponding to the same point in $\M_k$. These maps compose in the natural way. We can then define the Merenkov slit carpet $\M$ as the inverse limit of the system
$$ \M_0 \xleftarrow{\pi_{1,0}} \M_1 \xleftarrow{\pi_{2,1}} \M_2 \xleftarrow{\pi_{3,2}} \dots,  $$
equipped with the metric 
$$ d(x,y) = \lim_{k\rightarrow \infty} d_k(x_k, y_k)$$
for $x=(x_k)$ and $y=(y_k)$ such that $\pi_{k}(x_k)=x_{k-1}$, and similarly for $y_k$. Note that this is the limit of a bounded, increasing sequence.

Alternatively, one could also define $\M$ as the Gromov-Hausdorff limit of the $(\M_k,d_k)$, or directly by removing all the slits from $[0,1]$ and taking a metric completion with respect to the path metric.

For each $k$, there is a natural $1$-Lipschitz projection $\pi_k \colon \M \rightarrow \M_k$, given by $\pi_k((x_i))=x_k$. We let $\pi=\pi_0\colon \M \rightarrow \M_0 \cong Q_0.$ Merenkov proves the following about $\pi$:
\begin{lemma}[\cite{Me}, Lemma 2.3]
The map $\pi\colon \M \rightarrow \M_0\cong Q_0$ is regular.
\end{lemma}

Merenkov uses this lemma and other properties of $\pi$ to show that $\M$ is Ahlfors $2$-regular. (See \cite[p. 370]{Me}.) 

We take the opportunity to introduce some further notation describing $\M$ that we will use below.

Each slit $s_{l,k}^n \subseteq Q_0$ has a pre-image in $\M$ under $\pi$ that is a topological circle. Set $m_{l,k}^n$ to be the midpoint of $s_{l,k}^n$ in $Q_0$, and let $m_{l,k}^{n,+}$ and $m_{l,k}^{n,-}$in $\M$ denote the two pre-images of $m_{l,k}^n$ under $\pi$. (There is one on each ``side''.) On the other hand, the top and bottom $t_{l,k}^n$ and $b_{l,k}^n$ have single pre-images under $\pi$. We denote those pre-images by $\tilde{t}_{l,k}^n$ and $\tilde{b}_{l,k}^n$, respectively.

Let 
$$\mathcal{V}_n\defeq \{(x,y) \in (Q_0)^2 : x=(l2^{-n},k2^{-n}), y = (l2^{-n},(k+1)2^{-n})\}$$ 
be the collection of all pairs of points defining vertical sides of the dyadic squares at level $n$, and $\mathcal{V} = \bigcup_{n\geq 0} \mathcal{V}$ the collection of all of these pairs at all scales. We say that a element of $\mathcal{V}$ is at level $n$ if $(x,y) \in \mathcal{V}_n$.

We use $\mathcal{V}$ to define a set of ``vertically adjacent'' pairs of points in the carpet $\M$. Let
$$ \mathcal{W}_n = \{ (v,w)\in\M^2 : (\pi(v),\pi(w))\in \mathcal{V} \text{ and } d(v,w) = |\pi(v)-\pi(w)| \}$$
and
$$\mathcal{W}=\cup_{n\geq 0} \mathcal{W}_n.$$
The point of the condition $d(v,w) = |\pi(v)-\pi(w)|$ in the definition of $\mathcal{W}$ is that, if $\pi(v)$ and $\pi(w)$ lie on a common slit in $Q_0$, then $v$ and $w$ must lie on the same ``side'' of that slit in $\M$. Note that, for instance, all of the following pairs are in $\mathcal{W}$:
$$ (\tilde{b}_{l,k}^n, m_{l,k}^{n,\pm}), (m_{l,k}^{n,\pm}, \tilde{t}_{l,k}^n).$$

\section{Proof of Theorem \ref{main-theorem}}

We begin with a few preliminary lemmas, and then give the proof of Theorem \ref{main-theorem}.

First, we observe that line segments in $Q_0$ between ``vertical pairs'' of points $(v,w)\in\mathcal{W}$ can be approximated by discrete paths which have a significant fraction of their length lying along slits.

\begin{lemma} \label{lem:approx} Let $(v,w)\in\mathcal{W}$. Then, for every $m > n$, there is a discrete path $(q_0, q_1, \dots, q_N)$ in $\M$ and a subset $G \subseteq \{0, \dots, N-1\}$ such that 
\begin{enumerate}
\item\label{eq:approx1} $q_0=v$, $q_N=w$,
\item\label{eq:approx1.5} $d(q_1,q_0)=d(q_N,q_{N-1})\leq 2^{-m-1}$,
\item\label{eq:approx2} $$\sum_{i=0}^{N-1} d(q_i,q_{i+1}) = 2^{-n} + 2^{-m},$$
\item\label{eq:approx3} $$\sum_{i \in G} d(q_i,q_{i+1}) = 2^{-n-1},$$
\item\label{eq:approx4} if $i \in G$, then $[\pi(q_i),\pi(q_{i+1})] = s_{a,b}^m$ for some $a,b$, and
\item\label{eq:approx5} if $i\notin G \cup \{0,N-1\}$, then $(q_i,q_{i+1})\in\mathcal{W}$.
\end{enumerate}
\end{lemma}

\begin{proof} 
Let $x=\pi(v)$ and $y=\pi(w)$, so that $(x,y)\in\mathcal{V}_n$.

There are $k,l$ so that $x=(l2^{-n},k2^{-n}), y = (l2^{-n},(k+1)2^{-n}$.  Let $m>n$  be arbitrary. We will first define a discrete path $(p_0, \dots, p_N)$ from $x$ to $y$ in $Q_0$.

We will take this path to be the original segment $[x,y]$ shifted by $2^{-m-1}$ to the either the left or right, and then discretized appropriately. We will now describe this discretization in detail.

The points $x$ and $y$ are dyadic of level $2^{-n}$. Therefore, no slit of level $\geq n$ intersects the horizontal lines through $x$ or $y$. Moreover, the assumption that $(v,w)\in \mathcal{W}$ implies that if, $v$ and $w$ lie on a common slit, then they lie on the same ``side'' of that slit.

Set $x^{\pm} = x\pm(2^{-m-1},0)$ and $y^{\pm}=y\pm(2^{-m-1},0)$. Let $v^{\pm}$ and $w^{\pm}$ be pre-images under $\pi$ of $x^{\pm}$ and $y^{\pm}$. (These pre-images are uniquely determined: since $m>n$, $x^{\pm}$ and $y^{\pm}$ cannot lie in the interior of a slit.) Then at least one pair of distances
\begin{equation}\label{eq:plus}
 d(v^+, v) \text{ and } d(w^+,w)
\end{equation}
or
\begin{equation}\label{eq:minus}
 d(v^-,v) \text{ and } d(w^-,w)
\end{equation}
are both equal to $2^{-m-1}$. Without loss of generality, we assume the former. (Note that if $l=0$, we take the first option, while if $l=2^n$ we take the second option.)

The line segment in $Q_0$ from $x^+$ to $y^+$, is half-covered by slits $s_{a,b}^m$. Let $N = 3\cdot 2^{m-n}+2$, and first define $p_0=x,p_N = y, p_{N-1} = y^+=(l2^{-n} + 2^{-m-1},(k+1)2^{-n})$. For the remaining $p_i$ ($i\in \{1, \dots, N-2\}$), we first represent $i=3s + j$, for $j=1,2,3$ and $s =0, \dots 2^{m-n}-1$. We then set

$$p_i = \left\{ \begin{array}{l r}
(l2^{-n} + 2^{-m-1}, k2^{-n} + s2^{-m}) & i = 3s+1, s =0, \dots 2^{m-n}-1 \\
 (l2^{-n} + 2^{-m-1}, k2^{-n} + s2^{-m}+2^{-m-2}) & i = 3s+2, s =0, \dots 2^{m-n}-1 \\
  (l2^{-n} + 2^{-m-1}, k2^{-n} + s2^{-m}+3\cdot 2^{-m-2}) & i = 3s+3, s =0, \dots 2^{m-n}-1 \\
\end{array} \right.$$

See Figure \ref{fig:approxpic}.

In other words, $(p_i)$ form a discrete path that starts at $x$, takes a step of size $2^{-m-1}$ to the right, proceeds up vertically with certain jumps until reaching the height of $y$, and then takes a step of size $2^{-m-1}$ to the left to reach $y$.

Observe that, for each $i\in \{1,\dots, N-1\}$, the point $p_i$ does not lie in the interior of any slit, and so has a unique pre-image $q_i\in \M$ under $\pi$. Moreover, we have
\begin{equation}\label{eq:qp}
d(q_i, q_{i+1}) = |p_i - p_{i+1}| \text{ for each } i\in\{0,\dots,N-1\} .
\end{equation}
Indeed, if $i\in\{1,\dots, N-2\}$ the step from $p_i$ to $p_{i+1}$ is in the vertical direction, in which case $d_k(p_i,p_{i+1})$ is the length of the segment $[p_i,p_{i+1}]$ for each $i$ and $k$. If we are in the horizontal step in which $i=0$ or $i=N-1$, equation \eqref{eq:qp} holds because of our understanding that both distances in \eqref{eq:plus} are $\leq 2^{-m-1}$. 

With this definition of $(q_i)$ and \eqref{eq:qp}, \eqref{eq:approx1} and \eqref{eq:approx1.5} are immediate. Item \eqref{eq:approx2} is also simple:
$$\sum_{i=0}^{N-1} d(q_i,q_{i+1}) = \sum_{i=0}^{N-1} |p_i-p_{i+1}| = 2^{-n} + 2^{-m},$$
since the $(p_i)$ form of a discrete vertical geodesic path of length $2^{-n}$, plus two horizontal steps of size $2^{-m-1}$.

We now set 
$$G = \{3s + 2 \ | \ s =0, \dots 2^{m-n}-1\}.$$
By the definition of $p_i$, it is easy to see that if $i\in G$, then $p_i=\pi(q_i)$ and $p_{i+1}=\pi(q_{i+1})$ are the bottom and top, respectively, of a slit $s^m_{a,b}$ in a cube of side length $2^{-m}$. This verifies \eqref{eq:approx4}.

Item \eqref{eq:approx5} is also simple by inspection: If $i\notin G\cup\{0,N-1\}$, then the formulae above for $p_i=\pi_i(q_i)$ and $p_{i+1}=\pi_i(q_{i+1})$ indicate that they are adjacent vertical corners of a dyadic square of side length $2^{-m-2}$, hence $(\pi(q_i),\pi(q_{i+1}))\in \mathcal{V}$. Moreover, as observed above, $d(q_i, q_{i+1})=|p_i-p_{i+1}|=|\pi(q_i)-\pi(q_{i+1})|$, which shows that $(q_i, q_{i+1})\in\mathcal{W}$.

Lastly, for item \eqref{eq:approx3}, we observe that
$$ \sum_{i \in G} |p_i-p_{i+1}|$$
is simply the total length of the slits in the vertical segment from $x+(2^{-m-1},0)$ to $y+(2^{-m-1},0)$, which is half the total length of that segment, and hence equal to $2^{-n-1}$. Item \eqref{eq:approx3} follows from this and \eqref{eq:qp}.

\begin{figure}
	\centering
		\includegraphics[scale=0.3]{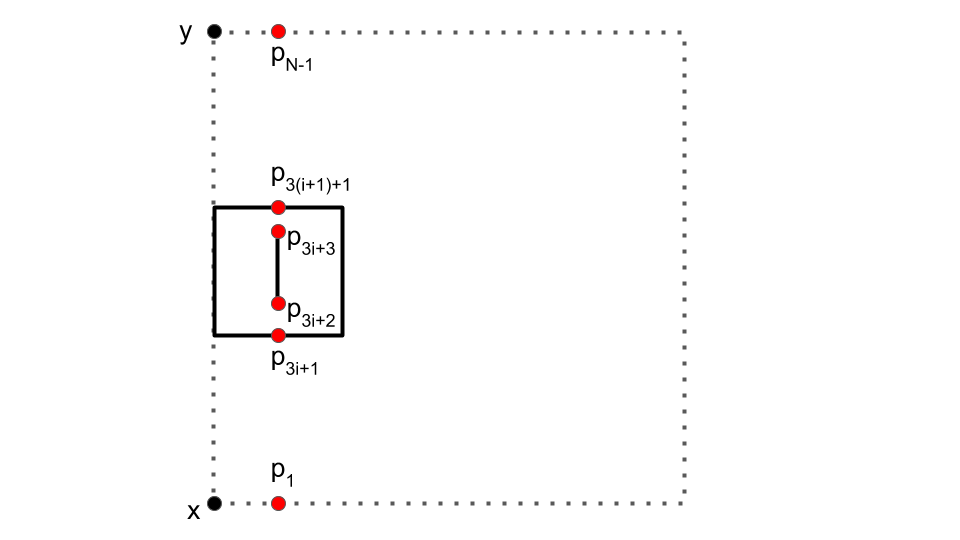}
		\caption{The points $p_i$}
	\label{fig:approxpic}
\end{figure}
\end{proof}

The next lemma concerns uniformly convex Banach spaces. In a uniformly convex Banach space, metric midpoints are unique. The uniform convexity property can be used to quantify this statement, as follows.

\begin{lemma} \label{lem:unifconvex} Suppose $x,y \in B$ are points, and $m\in B$ is an additional point. Then for every $\epsilon>0$, there exists an $\eta>0$ so that either

\begin{enumerate}
\item $$\max\{ \|x-m\|,\|y-m\|\} \geq \frac{1+\eta}{2}\|x-y\|$$ or
\item $$\left\|m-\frac{y+x}{2}\right\|  \leq \epsilon \|x-y\|.$$
\end{enumerate}
\end{lemma}

\begin{proof}
By translation and scaling, we can take $x=0$ and $\|y\|=1$. Apply the uniform convexity condition to $\epsilon$ to obtain a $\delta>0$. Let $\eta=\frac{1}{2}\min\{\epsilon,\delta\}.$ Suppose that the first property fails with this choice of $\eta$. Then by the triangle inequality we get

$$\frac{1-\eta}{2}\leq \|m\| \leq \frac{1+\eta}{2}$$
and
$$\frac{1-\eta}{2} \leq \|y-m\| \leq  \frac{1+\eta}{2}.$$
Define $\xi^1 = \frac{m}{\|m\|}$, and $\xi^2 = \frac{y-m}{\|y-m\|}$.  Note that
$$ \|\xi^1 - 2m\| \leq \eta \text{ and } \|\xi^2 - 2(y-m)\| \leq \eta.$$
Therefore,
$$\left\|\frac{\xi^2 + \xi^1}{2}\right\| \geq 1-\eta > 1-\delta.$$
Consequently, from the uniform convexity, $\|\xi^1-\xi^2\| \leq \epsilon$, and so $\|2m-2(y-m)\| \leq 2\eta + \epsilon$. In that case,
$$\|m-(y/2)\| = \frac{1}{4} \|2m-2(y-m)\|\leq \frac{\eta}{2}+\frac{\epsilon}{4} \leq \epsilon,$$
which gives the second possibility, as desired.
\end{proof}

On the other hand, the slit carpet $\M$ does not have unique midpoints, as the slits can be traversed on both ``sides''. The following lemma is immediate from the definition of $\M$.

\begin{lemma} \label{lem:midpoint} Suppose $s=s^{n}_{l,k} = [b_{l,k}^n,t_{l,k}^n]$ is any slit in $Q_0$. Recall the four associated points in $\M$ on $\pi^{-1}(s)$, which we called $\tilde{t}_{l,k}^n$, $\tilde{b}_{l,k}^n$, $m_{l,k}^{\pm,n}$.

Then
$$d(\tilde{t}^n_{l,k},\tilde{b}^{n}_{l,k}) = 2^{-n-1}$$
$$d(m^{n,\pm}_{l,k},\tilde{t}^n_{l,k}) = d(m^{n,\pm}_{l,k},\tilde{b}^n_{l,k}) = 2^{-n-2} = d(\tilde{t}^n_{l,k},\tilde{b}^{n}_{l,k})/2$$ and

$$d(m^{n,+}_{l,k}, m^{n,-}_{l,k}) \geq 2^{-n-1} = d(\tilde{t}^n_{l,k},\tilde{b}^{n}_{l,k}).$$
\end{lemma}

We are now ready to prove Theorem \ref{main-theorem}, that $\M$ admits no bi-Lipschitz embedding into any uniformly convex Banach space $B$. As noted above, the argument is heavily inspired by the framework of Burago and Kleiner \cite{BK} and also an argument of Laakso \cite{Laakso}.

\begin{proof}[Proof of Theorem \ref{main-theorem}]
Suppose that $B$ is a uniformly convex Banach space, and $f\colon \M \to B$ is a bi-Lipschitz mapping. Let $(b,L)$ be the bi-Lipschitz constants of $f$.

Recall the definitions of $\mathcal{W}$ and $\mathcal{W}_n$ from Section \ref{sec:notation} above.
 
 We define the maximal vertical distortion of $f$ by
 
 $$L_v = \sup_{(x,y) \in \mathcal{W}} \frac{|f(x)-f(y)|}{d(x,y)}.$$
 
Note that $L_v$ is bounded above by the Lipschitz constant $L$ of $f$, and below by the lower-Lipschitz constant $b>0$ of $f$. We proceed to derive a contradiction. 

Fix $\epsilon = \frac{b}{4L}$ and apply Lemma \ref{lem:unifconvex} to obtain a corresponding $\eta>0$. Next, choose $\eta'>0$ so that $\frac{L_v}{L_v-\eta'} < 1+\eta$.

 Choose a pair $(x,y) \in \mathcal{W}_n$ so that
\begin{equation}\label{eq:nearmax}
L_v-\frac{\eta'}{4} \leq \frac{|f(x)-f(y)|}{d(x,y)}.
\end{equation}
 
 In particular, $|f(x)-f(y)| \geq 2^{-n}L_v + \eta' 2^{-n-2}$.
 
Then, choose $m>n$ large enough so that
\begin{equation}\label{eq:mbound}
(2L+L_v)2^{-m} < \eta'2^{-n-2}.
\end{equation}

Using Lemma \ref{lem:approx} (with $v=x$ and $w=y$), we can find a discrete path $q_0, \dots, q_N$ from $x$ to $y$ in $\M$ with the properties in the statement.

By Lemma \ref{lem:approx} \eqref{eq:approx5}, for $i \not\in G\cup\{0,N-1\}$, we have $(q_i, q_{i+1})\in\mathcal{W}$, and so
$$|f(q_{i+1})-f(q_i)| \leq L_v d(q_{i+1},q_i).$$

We now argue that there exists an $i \in G$ so that 
\begin{equation}\label{eq:goodpt}
|f(q_{i+1})-f(q_{i})| \geq (L_v-\eta') d(q_{i+1},q_i).
\end{equation}
Suppose that this was not the case. In that case, using properties \eqref{eq:approx1.5}, \eqref{eq:approx2}, \eqref{eq:approx3}, and \eqref{eq:approx5} from Lemma \ref{lem:approx}, we have
\begin{align*}
|f(x)-f(y)| &\leq |f(q_1) - f(q_0) | + \sum_{i=1}^{N-2} |f(q_{i+1})-f(q_{i})| + |f(q_N) - f(q_{N-1})|\\
&\leq L2^{1-m} + \sum_{i \in G} |f(q_{i+1})-f(q_{i})| + \sum_{i \not\in G \cup \{0,N-1\}} |f(q_{i+1})-f(q_{i})| \\
&\leq L2^{1-m} + (L_v-\eta') \sum_{i \in G} d(q_{i+1},q_i) + L_v\sum_{i \not\in G}  d(q_{i+1},q_i)  \\
&= L2^{1-m} + L_v\sum_{i=0}^{N-1} d(q_{i+1},q_i) -\eta'\sum_{i \in G}  d(q_{i+1},q_i)  \\
&\leq L2^{1-m} +  L_v (2^{-n} + 2^{-m}) - \eta' 2^{n-1}\\
&= 2^{-n}\left(L_v - \frac{\eta'}{2}\right) + 2^{-m}(2L+L_v)\\
&< \left(L_v - \frac{\eta'}{4}\right)d(x,y),
\end{align*}
where in the last line we used \eqref{eq:mbound}. However, this contradicts \eqref{eq:nearmax}. Therefore, \eqref{eq:goodpt} holds for some $i\in G$.
 
We thus have an $i \in G$ so that $|f(q_{i+1})-f(q_{i})| \geq (L_v-\eta') d(q_{i+1},q_i)$. By Lemma \ref{lem:approx} \eqref{eq:approx4}, this coincides with a slit $s^{m}_{k,l}$, with $q_{i+1}=\tilde{t}^{m}_{k,l}$ and $q_i = \tilde{b}^{m}_{k,l}$. Now, consider the two pre-images $m^{m,\pm}_{k,l}$ under $\pi$ of the mid-point of this slit, as defined near the end of Section \ref{sec:notation}.

Set $M^{\pm} = f(m^{m,\pm}_{k,l})$, $X=f(q_i)$, and $Y=f(q_{i+1})$ in $B$. Since the pairs $(m^{m,\pm}_{k,l}, q_{i+1})$ and $(q_i, m^{m,\pm}_{k,l})$ are in $\mathcal{W}$, we have the following bounds:

\begin{align*}
\|M^{\pm} - X\| &\leq L_v d(m^{m,\pm}_{k,l}, q_i)  \\
&= L_v d(q_{i+1},q_i)/2 \\
&\leq \frac{L_v}{2(L_v -\eta')}\|Y-X\|\\
&< \frac{1+\eta}{2}\|Y-X\|.
\end{align*}

Similarly, we can conclude that 
$$\|M^{\pm}-Y\| < \frac{1+\eta}{2}\|Y-X\|.$$ 

Consequently, from Lemma \ref{lem:unifconvex} we have that 
$$ \left\|M^{\pm}- \frac{X+Y}{2}\right\| \leq \epsilon \|Y-X\|$$
and hence that
$$ \|M^{+}-M^{-}\| \leq 2\epsilon\|Y-X\| \leq 2\epsilon L d(q_{i+1},q_i) = \frac{b}{2}d(q_{i+1},q_i),$$
using our choice of $\epsilon = b/4L$.

On the other hand, since $f$ is bi-Lipschitz with lower Lipschitz constant $b$, we also have by Lemma \ref{lem:midpoint} that
$$ \|M^{+}-M^{-}\| \geq b d(m^{m,+}_{k,l}, m^{m,-}_{k,l}) \geq bd(q_{i+1},q_i)>0.$$ 

This is a contradiction.
\end{proof}

\bibliography{slitcarpetbib}{}

\begin{thebibliography}{10}

\bibitem{BK}
D.~Burago and B.~Kleiner.
\newblock Separated nets in {E}uclidean space and {J}acobians of bi-{L}ipschitz
  maps.
\newblock {\em Geom. Funct. Anal.}, 8(2):273--282, 1998.

\bibitem{Ch}
J.~Cheeger.
\newblock Differentiability of {L}ipschitz functions on metric measure spaces.
\newblock {\em Geom. Funct. Anal.}, 9(3):428--517, 1999.

\bibitem{CK_Banach}
J.~Cheeger and B.~Kleiner.
\newblock On the differentiability of {L}ipschitz maps from metric measure
  spaces to {B}anach spaces.
\newblock In {\em Inspired by {S}. {S}. {C}hern}, volume~11 of {\em Nankai
  Tracts Math.}, pages 129--152. World Sci. Publ., Hackensack, NJ, 2006.

\bibitem{CK}
J.~Cheeger and B.~Kleiner.
\newblock Realization of metric spaces as inverse limits, and bilipschitz
  embedding in {$L_1$}.
\newblock {\em Geom. Funct. Anal.}, 23(1):96--133, 2013.

\bibitem{CK_PI}
J.~Cheeger and B.~Kleiner.
\newblock Inverse limit spaces satisfying a {P}oincar\'e inequality.
\newblock {\em Anal. Geom. Metr. Spaces}, 3:15--39, 2015.

\bibitem{CKS}
J.~Cheeger, B.~Kleiner, and A.~Schioppa.
\newblock Infinitesimal structure of differentiability spaces, and metric
  differentiation.
\newblock {\em Anal. Geom. Metr. Spaces}, 4(1):104--159, 2016.

\bibitem{DS}
G.~David and S.~Semmes.
\newblock {\em \normalfont ``{F}ractured fractals and broken dreams''},
  volume~7 of {\em Oxford Lecture Series in Mathematics and its Applications}.
\newblock The Clarendon Press, Oxford University Press, New York, 1997.

\bibitem{GCD}
G.~C. David.
\newblock Tangents and rectifiability of {A}hlfors regular {L}ipschitz
  differentiability spaces.
\newblock {\em Geom. Funct. Anal.}, 25(2):553--579, 2015.

\bibitem{GCDLip}
G.~C. David.
\newblock On the {L}ipschitz dimension of {C}heeger-{K}leiner.
\newblock {P}reprint, 2019.
\newblock ar{X}iv:1908.04421.

\bibitem{GCDKin}
G.~C. David and K.~Kinneberg.
\newblock Lipschitz and bi-{L}ipschitz maps from {PI} spaces to {C}arnot
  groups.
\newblock {P}reprint, 2017.
\newblock ar{X}iv:1711.03533.

\bibitem{Fitzpatrick}
S.~Fitzpatrick.
\newblock Differentiation of real-valued functions and continuity of metric
  projections.
\newblock {\em Proc. Amer. Math. Soc.}, 91(4):544--548, 1984.

\bibitem{Hak}
H.~Hakobyan.
\newblock Quasisymmetrically co-{H}opfian {S}ierpi\'nski spaces and {M}enger
  {C}urve.
\newblock {P}reprint, 2017.
\newblock ar{X}iv:1712.00526.

\bibitem{He}
J.~Heinonen.
\newblock {\em \normalfont ``{L}ectures on analysis on metric spaces''}.
\newblock Universitext. Springer-Verlag, New York, 2001.

\bibitem{HS}
J.~Heinonen and S.~Semmes.
\newblock Thirty-three yes or no questions about mappings, measures, and
  metrics.
\newblock {\em Conform. Geom. Dyn.}, 1:1--12 (electronic), 1997.

\bibitem{Laakso}
T.~J. Laakso.
\newblock Plane with {$A_\infty$}-weighted metric not bi-{L}ipschitz embeddable
  to {${\Bbb R}^N$}.
\newblock {\em Bull. London Math. Soc.}, 34(6):667--676, 2002.

\bibitem{Me}
S.~Merenkov.
\newblock A {S}ierpi\'{n}ski carpet with the co-{H}opfian property.
\newblock {\em Invent. Math.}, 180(2):361--388, 2010.

\bibitem{MW}
S.~Merenkov and K.~Wildrick.
\newblock Quasisymmetric {K}oebe uniformization.
\newblock {\em Rev. Mat. Iberoam.}, 29(3):859--909, 2013.

\bibitem{MLW}
H.~Movahedi-Lankarani and R.~Wells.
\newblock On bi-{L}ipschitz embeddings.
\newblock {\em Port. Math. (N.S.)}, 62(3):247--268, 2005.

\bibitem{Preiss}
D.~Preiss.
\newblock Differentiability of {L}ipschitz functions on {B}anach spaces.
\newblock {\em J. Funct. Anal.}, 91(2):312--345, 1990.

\end{thebibliography}
\bibliographystyle{plain}

\end{document}